\documentclass[12pt,reqno]{article}

\usepackage[usenames]{color}
\usepackage{amssymb}
\usepackage{amsmath}
\usepackage{amsthm}
\usepackage{tikz}
\usetikzlibrary{automata,positioning,arrows}

\usepackage[colorlinks=true,
linkcolor=webgreen,
filecolor=webbrown,
citecolor=webgreen]{hyperref}

\definecolor{webgreen}{rgb}{0,.5,0}
\definecolor{webbrown}{rgb}{.6,0,0}

\usepackage{fullpage}

\newcommand{\seqnum}[1]{\href{https://oeis.org/#1}{\rm \underline{#1}}}

\theoremstyle{plain}
\newtheorem{theorem}{Theorem}
\newtheorem{corollary}[theorem]{Corollary}
\newtheorem{lemma}[theorem]{Lemma}
\newtheorem{proposition}[theorem]{Proposition}

\theoremstyle{definition}
\newtheorem{definition}[theorem]{Definition}
\newtheorem{example}[theorem]{Example}

\theoremstyle{remark}
\newtheorem{remark}[theorem]{Remark}

\begin{document}

\begin{center}
\vskip 1cm{\LARGE\bf Meta-Automatic Sequences}
\vskip 1 cm
\large

John M.\ Campbell\\
Department of Mathematics and Statistics\\
 Dalhousie University\\ 
 Halifax, NS B3H 4R2\\ 
 Canada \\
\href{mailto: jh241966@dal.ca}{\tt jh241966@dal.ca}\\
\vskip .2 in
 Beno\^{\i}t Cloitre \\ 
 19, rue Louise Michel\\
 92300 Levallois-Perret\\
 France\\
\href{mailto:benoit7848c@yahoo.fr}{\texttt{benoit7848c@yahoo.fr}} \\ 

\end{center}

\vskip .2 in

\begin{abstract}
 \emph{Nested} (or \emph{meta-Fibonacci}) recurrences, such as the recurrence used to define Hofstadter's $Q$-sequence, along with the 
 \emph{digit-based} recurrences that underlie automatic sequences are of interest from both number-theoretic and combinatorial points 
 of view. In this direction, Allouche and Shallit showed how the frequency sequence of a variant of the $Q$-sequence is $2$-automatic. 
 This inspires us to introduce what may be seen as a natural combination of the recurrences for meta-Fibonacci and automatic sequences, 
 by introducing the concept of a \emph{meta-automatic sequence}. We exhibit two binary meta-automatic sequences $\mathcal{M}_1$ 
 and $\mathcal{M}_{2}$ whose defining recurrences do not satisfy the Allouche--Shallit automaticity criterion directly, and this is 
 formalized in our paper. For each of these integer sequences $\mathcal{M}_{1}$ and $\mathcal{M}_{2}$, we prove explicit DFAO 
 evaluations, together with 4-uniform morphisms, and we also consider the factor complexities of these sequences. 
 \end{abstract}
 
\section{Introduction}
 A \emph{nested recurrence} is a recurrence in which a term $a(n)$ is defined using values of $a$ at indices that themselves depend on 
 earlier values of the sequence. The classical example is Hofstadter's $Q$-sequence \cite[p.\ 137]{Hofstadter1979}, indexed in the 
 On-Line Encyclopedia of Integer Sequences \cite{oeis} as \seqnum{A005185}, defined by $Q(1)=Q(2)=1$ and
\begin{equation}\label{displayQrec} 
 Q(n) = Q(n-Q(n-1)) + Q(n-Q(n-2)) 
\end{equation}
 for $n>2$. Recurrences as in \eqref{displayQrec} are also known as \emph{meta-Fibonacci recurrences}. Other examples of integer 
 sequences satisfying nested recurrences include the sequence \seqnum{A244477} studied by Golomb \cite{Golomb1990}. Many 
 problems concerning Hofstadter's $Q$-sequence and its generalizations remain open to this day. Notably, it is not known whether or 
 not $Q(n)$ is even \emph{defined} for all $n>0$.

 In contrast to nested recurrences, \emph{digit-based recurrences} (also called \emph{divide-and-conquer recurrences}) define 
 $a(k^{t}n+j)$ for $j\in\{0,1,\ldots,k^{t}-1\}$
 as a function of previously computed values $a(k^{t'}n'+j')$ with $t'<t$ and
 $j'\in\{0,1,\ldots,k^{t'}-1\}$. Digit-based recurrences can be used to define \emph{automatic
 sequences}~\cite{AlloucheShallit2003,Cobham1972}, 
 which appear in many areas of number
 theory, computer science, and combinatorics.

 Allouche and Shallit \cite{AlloucheShallit2012} showed that the frequency sequence of a variant of Hofstadter's $Q$-sequence satisfies a 
 digit-based recurrence, and, moreover, is $2$-automatic. They also proved an automaticity criterion 
 \cite[Theorem~2.2]{AlloucheShallit2012} reviewed in Section~\ref{subsectionpre} below. This inspires us to introduce the concept of a 
 \emph{meta-automatic sequence}, which may, informally, be thought of as ``combining'' the concepts of meta-Fibonacci and 
 automatic sequences. This is formalized below. For example, in contrast to a rule such as $a(4n+2) = a(2n)$ used to define an 
 automatic sequence, we consider rules such as
 \begin{equation}\label{displayintroex} 
 a(4 n + 2) = a(2n + 1 - a(n)),
 \end{equation}
 in which the argument of $a$ on the right-hand side depends on the value~$a(n)$. A recurrence of the form suggested in 
 \eqref{displayintroex} would not, in general, satisfy the Allouche--Shallit automaticity criterion, since the argument on the right is not 
 of the required polynomial form. 

 The idea of combining automatic and meta-Fibonacci recurrences is motivated by past work on meta-Fibonacci sequences, in addition to 
 the above referenced work by Allouche and Shallit \cite{AlloucheShallit2012}. We point to references on the
 {H}ofstadter {$Q$}-sequence~\cite{Fox2020,Fox2016Quasipolynomial,Hendel2015,Pinn1999},
 on variants and generalizations of the
 {H}ofstadter {$Q$}-sequence~\cite{Alkan2018,AlkanFoxAybar2017,BalamohanKuznetsovTanny2007,Dekking2023,Fox2024,Tanny1992}, 
 and on meta-Fibonacci sequences~\cite{CaiTanny2008,CallaghanChewIIITanny2005,DaltonRahmanTanny2011,
DeugauRuskey2006,Emerson2006,Fox2016Linear,IsgurRahman2011,
IsgurReissTanny2009,JacksonRuskey2006,RuskeyDeugau2009,SobolewskiUlas2023}.

\subsection{Outline}
 A formal definition of a meta-automatic sequence is given in Section~\ref{sectionmetamain}. A meta-automatic recurrence may, after 
 simplification, reduce to a digit-based recurrence, in which case the Allouche--Shallit criterion applies directly. 
 We formalize, in this paper, how meta-automatic sequences are not necessarily reducible in this way. 

 A binary sequence $a:\mathbb{N}_0\to\{0,1\}$ is said to be \emph{balanced} if
\begin{equation}\label{eq:balance} 
 a(2n)+a(2n+1)=1\qquad (n\ge 0).
\end{equation}
 The binary sequences under consideration in this paper are, for the most part, balanced. In a related way, much of this paper makes 
 use of the additive operation underlying the field
 $\mathbb{F}_2 = \{0,1\}$, letting $\oplus$ denote this
 operation, which may be referred to as \emph{XOR}, with $0\oplus 0 = 
 1\oplus 1 = 0$ and $0\oplus 1 = 1\oplus 0 = 1$. For balanced binary sequences, Lemma~\ref{lem:balance-xor} gives the identity 
 $$ a(2n+1-a(n)) = a(2n+1)\oplus a(n). $$ We prove that balanced, binary, meta-automatic sequences are, for the cases 
 covered in Lemma~\ref{lem:dyadic}, ultimately periodic. To obtain non-periodic examples, we therefore work with base-$4$
 recurrences. We consider:

\begin{itemize}

 \item The integer sequence $\mathcal{M}_{1}$ (\seqnum{A392736}), defined via a base-$4$ recurrence in which one of the rules in 
 the recurrence is nested in the 
 sense outlined above. After applying Lemma~\ref{lem:balance-xor} below,
 the pair $(\mathcal{M}_{1}(n), b(n))$, for $b(n)$ defined in~\eqref{eq:b1} below, yields a $4$-state DFAO
 (Section~\ref{sec:meta1}).

\item The integer sequence $\mathcal{M}_{2}$ (\seqnum{A391614}), defined via a base-$4$ recurrence in which \emph{both} of the rules 
 in the recurrence are nested in the sense outlined above. After applying Lemma~\ref{lem:balance-xor}, the pair
 $\bigl(\mathcal{M}_{2}(n),\mathcal{M}_{2}(2n+\mathcal{M}_{2}(n))\bigr)$ obeys an affine
 rule over $\mathbb{F}_2^2$ with linear part
 $\bigl(\begin{smallmatrix}0 & 1\\1&1\end{smallmatrix}\bigr)$. The sequence satisfies
 $\mathcal{M}_{2}(n) = \mathbf{t}(q(n))$, where $q$ is an explicit bit-masking operator
 and $\mathbf{t}$ is the Thue--Morse sequence (Theorem~\ref{thm:M2-tq}).

\end{itemize}

 Although $\mathcal{M}_{1}$ and $\mathcal{M}_{2}$ both have $4$-state DFAOs, we have that $\mathcal{M}_{2}$ admits the evaluation 
 $\mathcal{M}_{2}(n) = \mathbf{t}(q(n))$, whereas $\mathcal{M}_{1}$ is described by means of the pair $(\mathcal{M}_{1}(n), b(n))$. 
 Their factor-complexity functions, studied in Section~\ref{sec:complexity}, also differ. For each of the sequences $\mathcal{M}_{1}$ and 
 $\mathcal{M}_{2}$, we give a base-$4$ DFAO (MSB-first) and an explicit $4$-uniform morphism, 
 and we also investigate the factor complexity functions for $\mathcal{M}_{1}$ and $\mathcal{M}_{2}$
 in Section \ref{sec:complexity}. 

 \subsection{Preliminaries}\label{subsectionpre} 
 For background on automatic sequences, we refer to the texts of Allouche and Shallit \cite{AlloucheShallit2003}
 and of Shallit \cite{Shallit2022}. We proceed with the following 
 definition, which may be used, as below, to define automatic sequences. 

\begin{definition} 
 Let $U = (U(n))_{n \geq 0}$ denote an infinite sequence. The \emph{$k$-kernel} of $(U(n))_{n \geq 0}$ \cite[p.\ 185]{AlloucheShallit2003} is 
 $$ K_{k}(U) = \left\{ \left( U_{k^{i}n+j} \right)_{n \geq 0} : \text{$i \geq 0$ and $0 \leq j < k^{i}$} \right\}. $$ 
\end{definition} 

 For $k \geq 2$, a \emph{$k$-automatic sequence} is a sequence $U = (U(n))_{n \geq 0}$ such that $K_{k}(U)$ is finite. Equivalently,
 the sequence $U$ is 
 $k$-automatic if there exists a DFAO such that, upon inputting the base-$k$ representation of $n$, this DFAO outputs $U(n)$. See 
 Allouche and Shallit's text \cite{AlloucheShallit2003} for details. A prototypical instance of an automatic sequence is the 
 \emph{Thue--Morse sequence} $\mathbf{t}:\mathbb{N}_0\to\{0,1\}$ defined so that $\mathbf{t}(0)=0$ and $$ \mathbf{t}(2n) = 
 \mathbf{t}(n),\qquad \mathbf{t}(2 n + 1) = 1 - \mathbf{t}(n)\qquad (n\ge 0). $$ 
 Equivalently, $\mathbf{t}(n)=s_2(n)\bmod 2$, where $s_2(n)$ is the sum of the binary digits of $n$, 
 noting that the balanced condition in \eqref{eq:balance} holds. 
 We introduce, in this paper, a variant of 
 an automaticity criterion due to Allouche and Shallit~\cite{AlloucheShallit2012}, which we now recall.

 Let $U = (U(n))_{n \geq 0}$ denote a sequence, and let $\alpha$ denote an integer. Adopting notation from Allouche and Shallit's work 
 \cite{AlloucheShallit2012}, for $n \geq -\alpha$, we write $U^{\alpha}$ in place of the sequence given by 
 $$ U^{\alpha} := U(n + \alpha). $$ Similarly, letting $q$ and $i$ and $j$ be positive integers, 
 we write $U_{q, i, j}$ in place of the sequence defined by 
\begin{equation}\label{Utriple} 
 U_{q, i, j}(n) = U\left( q^{i} n + j \right) 
\end{equation}
 for $n \geq 0$. Moreover, for integers $q \geq 2$ and $t \geq 1$, we write $s(q, t) = \frac{1-q^{t+1}}{1-q}$ in place of the number of 
 pairs $(i, j)$ of integers $i$ and $j$ such that $0 \leq i \leq t$ and $0 \leq j \leq q^{i} - 1$. Fix an ordering of these pairs $(i,j)$, with $(0, 
 0)$ appearing first, and write the corresponding sequences as
 $U_1=U,\,U_2,\ldots,\,U_{s(q,t)}$. The choice of ordering plays no role in what follows. 
 We also let $m(q, t) = q^{t+1} - 1$.

\begin{theorem}\label{AScrit} 
 (Allouche and Shallit, 2012) Let the terms of $(U(n))_{n \geq 0}$ be in a finite set $\mathcal{A}$, and let $q\geq 2$ denote an integer. 
 Then $U$ is $q$-automatic if there are nonnegative integers $t$, $a$, $b$, and $n_0$ together with a family $\{ f_{j} : j = 0, 1, \ldots, 
 m(q, t) \}$ of functions from $\mathcal{A}^{a + b + s(q, t)}$ to $\mathcal{A}$ such that 
\begin{multline*}
 U\big( q^{t+1} n + j \big) = 
 f_{j}\big( U^{-a}(n), \ldots, U^{-1}(n), U^{0}(n), U^{1}(n), \ldots, \\ U^{b}(n), U_{2}(n), 
 U_{3}(n), \ldots, U_{s(q, t)}(n) \big)
\end{multline*}
 for all $n \geq n_0$ and for all $j \in \{ 0, 1, \ldots, m(q, t) \}$ \cite{AlloucheShallit2012}. 
\end{theorem}

\section{Meta-automatic sequences}\label{sectionmetamain} 
 Let $q>0$, $i_0>0$, and $j_0\ge 0$ be integers. For some $v\ge 0$, and for indices $k\in\{1,2,\ldots,v\}$, let $c_k$, $i_k>0$, and $j_k\ge 
 0$ be integers, with $i_k<i_0$ for all $k$. By analogy with~\eqref{Utriple}, and assuming that the codomain of $U$ is contained in 
 $\mathbb{N}_0$, we write $$ U_{q, i_0, j_0, c_1, i_1, j_1, \ldots, c_v, i_v, j_v}(n) = U\Big(q^{i_0} n + 
 j_0+c_1 U(q^{i_1}n+j_1)+\cdots+c_v U(q^{i_v}n+j_v)\Big). $$
 The following definition may be thought of as formalizing 
 the intuition behind recurrences as in \eqref{displayintroex}. 

\begin{definition}\label{metaauto} 
 Suppose that each term of $(U(n))_{n\ge 0}$ is in a finite subset $\mathcal{A}$ of $\mathbb{N}_0$, and let $q\ge 2$ be an 
 integer. The sequence $U$ is \emph{$q$-meta-automatic} if there exist nonnegative integers $t$, $a$, $b$, $n_0$ together with a family
 $\{f_j : j=0,1,\ldots,m(q,t)\}$ of functions to $\mathcal{A}$, with domains as below,
 such that, for all $n\ge n_0$ and all $j\in\{0,1,\ldots,m(q,t)\}$, the value
 $U(q^{t+1}n+j)$ equals $f_j$ evaluated at the tuple
\begin{multline*}
 \Big( U^{-a}(n),\,\ldots,\,U^{-1}(n),\,U^{0}(n),\,U^{1}(n),\,\ldots,\,U^{b}(n),\\
 U_{q,i_0^{(1)},j_0^{(1)},c_1^{(1)},i_1^{(1)},j_1^{(1)},\ldots,c_{v_1}^{(1)},i_{v_1}^{(1)},j_{v_1}^{(1)}}(n),\,\ldots,\\
 U_{q,i_0^{(w)},j_0^{(w)},c_1^{(w)},i_1^{(w)},j_1^{(w)},\ldots,c_{v_w}^{(w)},i_{v_w}^{(w)},j_{v_w}^{(w)}}(n)\Big)
\end{multline*}
 for some $w\ge 0$ and integers $i_{k_1}^{(k_2)}\le t$, $j_{k_1}^{(k_2)}$, and
 $c_{k_1}^{(k_2)}$, where the domain of $f_j$ is $\mathcal{A}^{a+b+w+1}$.
\end{definition}

\begin{example}
 Define $(\mathcal{T}(n))_{n\ge 0}$ by $\mathcal{T}(0)=0$ and
\begin{align*}
 \mathcal{T}(2n) & = \mathcal{T}(n - \mathcal{T}(n-1)) \ \text{for $n \geq 1$, and} \\
 \mathcal{T}(2n+1) & = 1 - \mathcal{T}(n) \ \text{for $n \geq 0$}.
\end{align*}
 Then $(\mathcal{T}(n))_{n\ge 0}$ is $2$-meta-automatic in the sense of
 Definition~\ref{metaauto}. For $n\ge 1$, this sequence agrees with the OEIS sequence
 \seqnum{A039982}, which was defined without nested recurrences.
\end{example}

\begin{lemma}\label{lem:balance-xor} 
 If $a$ is balanced, then $$ a(2n+a(n)) = a(2n)\oplus a(n) \quad\text{and}\quad a(2n+1-a(n)) = a(2n+1)\oplus a(n), $$ for all $n\ge 0$. 
\end{lemma}

\begin{proof}
 The first identity follows from~\eqref{eq:balance} by case analysis on $a(n)\in\{0,1\}$, and similarly for the second.
\end{proof}

\subsection{Base-4 recurrences}
 We now record a result that motivates the use of base-$4$ recurrences in the sequel. We let $$ a(2n)=a(n+a(n))\qquad 
 \text{or}\qquad a(2n)=a(n+1-a(n)), $$ and we let $a(2n+1)=1-a(2n)$. In both cases, we obtain ultimate periodicity, as we now show.

\begin{lemma}\label{lem:dyadic} 
 Let $a:\mathbb{N}_0\to\{0,1\}$ be balanced.
\begin{enumerate}
\renewcommand{\labelenumi}{(\alph{enumi})}
\item If $a(2n)=a(n+a(n))$ for all $n\ge 0$, then $a(0)=0$ and 
 $a$ is ultimately $2$-periodic, with 
 $$ a(2n)=0,\qquad a(2n+1)=1\qquad (n\ge 2). $$

\item If $a(2n)=a(n+1-a(n))$ for all $n\ge 0$, then $a(0)=1$ and $a$ is ultimately $2$-periodic, with 
$$ a(2n)=1,\qquad a(2n+1)=0\qquad (n\ge 2). $$
\end{enumerate}
\end{lemma}
 
\begin{proof}
 (a) Taking $n=0$ gives $a(0)=a(a(0))$. If $a(0)=1$, then balancedness gives us that $a(1)=0$, and hence $a(0)=a(1)=0$, 
 a contradiction. Thus $a(0) = 0$ and $a(1)=1$. 

 Write $b_m=a(2m)$, so that balancedness gives that $a(2m+1)=1-b_m$. Using $a(2n)=a(n+a(n))$ with $n=2m$, we obtain $$ b_{2 
 m} = a(4 m) = a(2 m + a(2 m)) = a(2 m + b_m) = 
\begin{cases}
a(2m) = b_m, & \text{if } b_m=0;\\
a(2m+1)=1-b_m=0, & \text{if } b_m=1.
\end{cases}
$$ So, we have that $b_{2m}=0$ for all $m\ge 0$.

 Now, using $a(2n)=a(n+a(n))$ with $n=2m+1$ and $a(2m+1)=1-b_m$, we get
\begin{multline*}
b_{2m+1}=a(4m+2)=a(2m+1+a(2m+1))=a(2m+2-b_m) = \\ 
\begin{cases}
a(2m+2)=b_{m+1}, & \text{if } b_m=0;\\
a(2m+1)=0, & \text{if } b_m=1.
\end{cases}
\end{multline*}
 We claim that $b_m=0$ for all $m\ge 2$. This holds for $m=2$, since $b_2 = 0$, and for $m=3$ because the above identity with $m = 
 1$ gives $b_3=b_2$ if $b_1=0$, and $b_3=0$ if $b_1=1$; in either case $b_3=0$ since $b_2=0$. Now let $m\ge 4$. If $m$ is even 
 then $b_m=b_{2r}=0$. If $m$ is odd, write $m=2r+1$ with $r\ge 1$. If $b_r=1$ then $b_m=0$; if $b_r=0$ then $b_m=b_{r + 1}$, and 
 $r + 1 < m$, so by induction $b_{r+1}=0$. Hence $b_m=0$ for all $m\ge 2$, i.e., $a(2n)=0$ for all $n\ge 2$. The balance condition in 
 \eqref{eq:balance} then gives $a(2n+1)=1$ for all $n\ge 2$.

 (b) This follows from (a) by complementing: if $a$ satisfies (b) then $\bar a(n)=1-a(n)$ satisfies (a), and vice versa.
\end{proof}

 Lemma \ref{lem:dyadic} leads us to consider base-$4$ recurrences such that the balance identity in \eqref{eq:balance} is preserved. 

\subsection{A non-nested case}
 Fix $a(0)=0$ and $a(1)=1$, and define
\begin{equation}\label{eq:4ary-skeleton} 
 a(4 n) = F(n), \quad a(4n+1)=1-F(n), 
 a(4n+2)=G(n), \quad a(4n+3)=1-G(n), 
\end{equation}
 so that $a(2n)+a(2n+1)=1$ holds automatically. The sequences studied in this paper arise as
 specializations of~\eqref{eq:4ary-skeleton}. We first consider a sequence $(\mathcal{Q}(n))_{n\ge 0}$
 whose recurrence is not nested, but which is used in the sequel.

\begin{definition}
The \emph{Thue--Morse Quarto} sequence $\mathcal{Q}:\mathbb{N}_0\to\{0,1\}$ (\seqnum{A298952}, 
 where we use the convention $\mathcal{Q}(0)=0$ rather than $a(0)=1$ as in the OEIS entry) is 
 defined by $\mathcal{Q}(0)=0$, $\mathcal{Q}(1)=1$, and
\begin{align*}
\mathcal{Q}(4n) &= \mathcal{Q}(n), & \mathcal{Q}(4n+1) &= 1-\mathcal{Q}(n), \\
\mathcal{Q}(4n+2) &= \mathcal{Q}(2n), & \mathcal{Q}(4n+3) &= 1-\mathcal{Q}(2n). 
\end{align*}
\end{definition}

 The sequence $\mathcal{Q}$ is $4$-automatic. It is computed by the $3$-state DFAO with states in $\{A, B, C\}$, initial state $A$, and 
 output $\tau(A)=\tau(B)=0$, $\tau(C)=1$, shown in
 Figure~\ref{fig:Q-dfao}. This DFAO corresponds to the $4$-uniform morphism
$$ A\mapsto ACBC,\quad B\mapsto BCCB,\quad C\mapsto CBBC, $$
 along with a coding such that $\tau(A)=\tau(B)=0$ and $\tau(C)=1$.

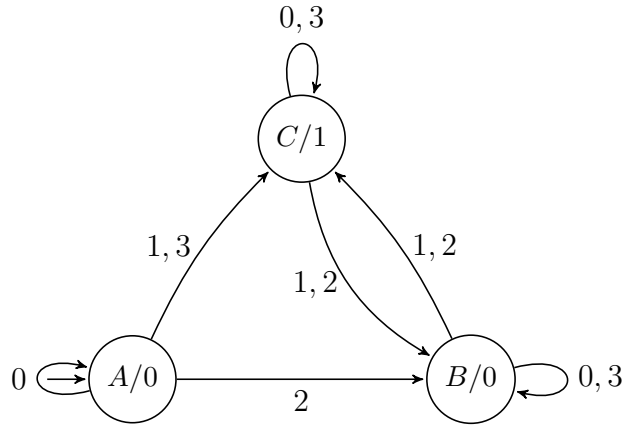
\begin{figure}[ht]
\centering
\begin{tikzpicture}[->, >=stealth', shorten >=1pt, semithick,
 every state/.style={minimum size=1.1cm, font=\small},
 initial text={}]
 \node[state,initial] (A) at (0,0) {$A/0$};
 \node[state] (B) at (4.5,0) {$B/0$};
 \node[state] (C) at (2.25,3.2){$C/1$};
 \path
 (A) edge[loop left] node[left] {$0$} (A)
 (B) edge[loop right] node[right] {$0,3$} (B)
 (C) edge[loop above] node[above] {$0,3$} (C)
 (A) edge node[below] {$2$} (B)
 (A) edge[bend left=10] node[left] {$1,3$} (C)
 (B) edge[bend right=10]node[right] {$1,2$} (C)
 (C) edge[bend right=25]node[left] {$1,2$} (B);
\end{tikzpicture}
\caption{DFAO for $\mathcal{Q}$ (MSB-first). State labels 
 are \emph{state}/\emph{output}. The initial state is $A$.}\label{fig:Q-dfao} 
\end{figure}

\begin{proposition}\label{prop:Q-rep} 
 Define $d: \mathbb{N}_0\to\{0,1\}$ so that $d(0)=0$ and $d(n) = \lfloor \log_2 n \rfloor \bmod 2$ for $n \geq 1$. 
 Then $$ \mathcal{Q}(n) = \mathbf{t}(n) \oplus d(n) \qquad (n \geq 0). $$
\end{proposition}

\begin{proof}
 The cases $n=0,1$ are immediate. For $n \geq 1$ and $r \in \{0,1,2,3\}$, since $\lfloor\log_2(4 n + r)\rfloor = 2 + \lfloor\log_2 n\rfloor$, 
 we have $d(4n+r) = d(n)$ for all $r$. We also use $\mathbf{t}(4n+r) = \mathbf{t}(n) \oplus \mathbf{t}(r)$, and we proceed
 by induction. 

\emph{Case $r=0$:} $\mathcal{Q}(4n)=\mathcal{Q}(n)=\mathbf{t}(n) 
 \oplus d(n)$ and $\mathbf{t}(4n)\oplus d(4n)=\mathbf{t}(n)\oplus d(n)$.

\emph{Case $r=1$:} $\mathcal{Q}(4n+1)=1\oplus\mathcal{Q}(n) = 
 1\oplus\mathbf{t}(n)\oplus d(n)$ and $\mathbf{t}(4n+1)
 \oplus d(4n+1)=(1\oplus\mathbf{t}(n))\oplus d(n)$.

\emph{Case $r=2$:} $\mathcal{Q}(4n+2)=\mathcal{Q}(2n)$. By induction, $\mathcal{Q}(2 
 n)=\mathbf{t}(2n)\oplus d(2n)=\mathbf{t}(n)\oplus(d(n) 
 \oplus 1)=1\oplus\mathbf{t}(n)\oplus d(n)$. Also $\mathbf{t}(4 n + 
 2)=\mathbf{t}(2n+1)=1\oplus\mathbf{t}(n)$, so $\mathbf{t}(4 n + 
 2)\oplus d(4n+2)=(1\oplus\mathbf{t}(n))\oplus d(n)$.

\emph{Case $r=3$:} $\mathcal{Q}(4n+3)=1\oplus\mathcal{Q}(2n)=\mathbf{t}(n)\oplus d(n)$, 
 and $\mathbf{t}(4n+3)\oplus d(4n+3)=\mathbf{t}(n)\oplus d(n)$.
\end{proof}

\section{\texorpdfstring{The sequence $\mathcal{M}_{1}$}{The sequence M1}}\label{sec:meta1} 

\begin{definition}
 Define $\mathcal{M}_{1}(0)=0$ and $\mathcal{M}_{1}(1)=1$. For $n\ge 1$, set
\begin{equation}\label{eq:M1-01} 
 \mathcal{M}_{1}(4n) = \mathcal{M}_{1}(n),\qquad \mathcal{M}_{1}(4n+1) = 1-\mathcal{M}_{1}(n).
\end{equation}
 For $n\ge 0$, set
\begin{equation}\label{eq:M1-23} 
 \mathcal{M}_{1}(4n+2) = \mathcal{M}_{1}(2n+1-\mathcal{M}_{1}(n)),\qquad
 \mathcal{M}_{1}(4n+3) = 1-\mathcal{M}_{1}(2n+1-\mathcal{M}_{1}(n)).
\end{equation}
\end{definition}

\begin{lemma}\label{lem:M1-wd} 
The sequence $\mathcal{M}_{1}(n)$ is well-defined for all $n\ge 0$ and satisfies $\mathcal{M}_{1}(2n)+\mathcal{M}_{1}(2n+1)=1$.
\end{lemma}

\begin{proof}
 Well-definedness follows by induction on $m$. For $m=4n+r$, the right-hand sides in \eqref{eq:M1-01} reduce to indices~$<m$; 
 for~\eqref{eq:M1-23}, the nested
 argument $2n+1-\mathcal{M}_{1}(n)\in\{2n,2n+1\}<4n+2\le m$. Balancedness follows
 from~\eqref{eq:4ary-skeleton}, since $\mathcal{M}_{1}(4n+1)=1-\mathcal{M}_{1}(4n)$
 and $\mathcal{M}_{1}(4n+3)=1-\mathcal{M}_{1}(4n+2)$.
\end{proof}

\begin{remark}
 Since $\mathcal{M}_{1}$ is balanced, Lemma~\ref{lem:balance-xor} gives
\[
\mathcal{M}_{1}(2n+1-\mathcal{M}_{1}(n)) = \mathcal{M}_{1}(2n+1)\oplus \mathcal{M}_{1}(n),
\]
 so~\eqref{eq:M1-23} is equivalent to the non-nested relation
 $\mathcal{M}_{1}(4n+2) = \mathcal{M}_{1}(2n+1)\oplus\mathcal{M}_{1}(n)$.
\end{remark}

\subsection{A 4-state DFAO}
 Define 
\begin{equation}\label{eq:b1} 
 b(n):=\mathcal{M}_{1}(2n+1-\mathcal{M}_{1}(n)), 
\end{equation}
 and consider the pair $s(n):=(\mathcal{M}_{1}(n), b(n))$. Informally, the following result gives us that the pair $s(4n+r)$ depends only 
 on $s(n)$, 
 for $r\in\{0, 1, 2, 3\}$ and for $n \geq 1$, disregarding the trivial $n = 0$ case whereby 
 $\mathcal{M}_{1}(0)=0$. 

 \begin{lemma}\label{lem:M1-closure} 
 Let $x=\mathcal{M}_{1}(n)$ and $y=b(n)$, and let $n \geq 1$. Then 
\begin{align}
 s(4n) & = (x, y), \label{s4n} \\ 
 s(4n + 1) & = (1 \oplus x,\; x \oplus y), \label{4np1} \\ 
 s(4 n + 2) & = (y,\; 1 \oplus x), \text{and} \label{4np2} \\
 s(4n + 3) & = (1 \oplus y,\; x \oplus y \oplus 1). \label{4np3} 
\end{align}
 \end{lemma}

\begin{proof}
 By Lemma~\ref{lem:M1-wd} and Lemma~\ref{lem:balance-xor}, we have 
 $y = b(n) = \mathcal{M}_{1}(2n+1)\oplus \mathcal{M}_{1}(n) = (1-\mathcal{M}_{1}(2n))\oplus x$, so that
\begin{equation*}
 \mathcal{M}_{1}(2n) = 1 \oplus x \oplus y, \qquad
 \mathcal{M}_{1}(2n+1) = x \oplus y.
\end{equation*}

\textbf{Case $r = 0$.} The first coordinate is $\mathcal{M}_{1}(4n) = \mathcal{M}_{1}(n) = x$. The latter coordinate satisfies 
 $b(4n) = \mathcal{M}_{1}(8n + 1 - \mathcal{M}_{1}(4n)) = \mathcal{M}_{1}(8n + 1 - x)$. If $x = 0$, then $b(4n) = \mathcal{M}_{1}(8 n + 
 1) = \mathcal{M}_{1}(4(2n)+1) = 1 - \mathcal{M}_{1}(2n) = 1 - (1 \oplus x \oplus y) = 1 - (1 \oplus y) = y$. If $x = 1$, then $b(4n) = \mathcal{M}_{1}(8n) = 
 \mathcal{M}_{1}(4(2n)) = \mathcal{M}_{1}(2n) = 1 \oplus x \oplus y = y$. In both cases, we find that 
 \eqref{s4n} holds. 

\textbf{Case $r = 1$.} The first coordinate is $\mathcal{M}_{1}(4n+1) = 1 - x = 1 \oplus x$. The latter coordinate satisfies 
 $b(4n+1) = \mathcal{M}_{1}(8n + 2 + x)$. If $x = 0$, then $b(4n+1) = \mathcal{M}_{1}(8 n + 2) = \mathcal{M}_{1}(4(2n)+2) = b(2n)$. 
 Using the property that $b(2n) = \mathcal{M}_{1}(4n + 1 - \mathcal{M}_{1}(2n))$, together with $\mathcal{M}_{1}(2n) = 1 \oplus y$, 
 this yields $b(2n) = \mathcal{M}_{1}(4n + y)$. So, if $y = 0$, then $\mathcal{M}_{1}(4n) = x = 0$, and if $y = 1$, then 
 $\mathcal{M}_{1}(4n+1) = 
 1 - x = 1$, so that $b(4n+1) = y$. Similarly, if $x = 1$, then $b(4n+1) = \mathcal{M}_{1}(8n+3) = \mathcal{M}_{1}(4(2n)+3) = 1 - b(2n)$.
 With $\mathcal{M}_{1}(2n) = y$, $b(2n) = \mathcal{M}_{1}(4n + 1 - y)$, we find that if $y = 0$, then $1 - \mathcal{M}_{1}(4n+1) = x = 
 1$, and if $y = 1$, then $1 - \mathcal{M}_{1}(4n) = 1-x = 0$, so that $b(4n+1) = 1 \oplus y$. Consequently, we obtain that 
 \eqref{4np1} holds. 

\textbf{Case $r = 2$.} The first coordinate is $\mathcal{M}_{1}(4n+2) = b(n) = y$. The latter coordinate satisfies $b(4 n + 
 2) = \mathcal{M}_{1}(8n + 5 - y)$. If $y = 0$, then $\mathcal{M}_{1}(8n+5) = \mathcal{M}_{1}(4(2n+1)+1) = 1 - \mathcal{M}_{1}(2n+1) = 
 1 \oplus (x \oplus y) = 1 \oplus x$. If $y = 1$, then $\mathcal{M}_{1}(8n+4) = \mathcal{M}_{1}(4(2n+1)) = \mathcal{M}_{1}(2n+1) = 
 x \oplus y = 1 \oplus x$. In both cases, we see that \eqref{4np2} holds. 

\textbf{Case $r = 3$.} The first coordinate is $\mathcal{M}_{1}(4n+3) = 1 - y = 1 \oplus y$. The latter coordinate 
 satisfies $b(4n+3) = 
 \mathcal{M}_{1}(8n + 6 + y)$. If $y = 0$, then $\mathcal{M}_{1}(8n+6) = \mathcal{M}_{1}(4(2n+1)+2) = b(2n+1)$.
Since $\mathcal{M}_{1}(2n+1) = x$, $b(2n+1) = \mathcal{M}_{1}(4n + 3 - x)$. So, if $x = 0$, then $\mathcal{M}_{1}(4n+3) = 1 \oplus y = 
 1$; if $x = 1$, then $\mathcal{M}_{1}(4n+2) = y = 0$; so $b(4n+3) = x \oplus y \oplus 1$. If $y = 1$, then $\mathcal{M}_{1}(8 n + 
 7) = \mathcal{M}_{1}(4(2n+1)+3) = 1 - b(2n+1)$. Since $\mathcal{M}_{1}(2n+1) = x \oplus 1$, we find that $b(2n+1) = \mathcal{M}_{1}(4 
 n + 2 + x)$. So, if $x = 0$, then $1 - \mathcal{M}_{1}(4n+2) = 1 - y = 0$; if $x = 1$, then $1 - \mathcal{M}_{1}(4n+3) = y = 1$, so that 
 $b(4n+3) = x \oplus y \oplus 1$. This gives us the desired relation in \eqref{4np3}. 
\end{proof}

\begin{theorem}\label{thm:M1-dfao} 
The sequence $\mathcal{M}_{1}$ is $4$-automatic and is computed (MSB-first, base $4$) by the DFAO 
 with states $\{0,1,2,3\}$, initial state $0$, output map $\tau(0)=\tau(2)=0$, $\tau(1)=\tau(3)=1$, and transition table
\begin{center}
\normalfont
\begin{tabular}{c|cccc}
$\delta$ & 0 & 1 & 2 & 3 \\ \hline
0 & 0 & 1 & 1 & 2 \\
1 & 1 & 2 & 3 & 0 \\
2 & 2 & 3 & 0 & 1 \\
3 & 3 & 0 & 2 & 3 \\
\end{tabular}
\end{center}
\end{theorem}

\begin{proof}
Lemma~\ref{lem:M1-closure} gives four $\mathbb{F}_2$-affine transition rules:
\[
\begin{aligned}
 (x,y) &\xrightarrow{\;0\;} (x,\; y), &
 (x,y) &\xrightarrow{\;1\;} (1 \oplus x,\; x \oplus y), \\
 (x,y) &\xrightarrow{\;2\;} (y,\; 1 \oplus x), &
 (x,y) &\xrightarrow{\;3\;} (1 \oplus y,\; x \oplus y \oplus 1).
\end{aligned}
\]
Encoding the states as
$\mathsf{0} = (0,1)$,
$\mathsf{1} = (1,1)$,
$\mathsf{2} = (0,0)$,
$\mathsf{3} = (1,0)$,
these yield the stated transition table. The initial state is $s(0) = (\mathcal{M}_{1}(0),\, b(0)) = (0,\, \mathcal{M}_{1}(1)) = 
 (0,1)$, corresponding to state~$\mathsf{0}$. The output reads the first coordinate $x = \mathcal{M}_{1}(n)$.
\end{proof}

\begin{figure}[ht]
\centering
\begin{tikzpicture}[->, >=stealth', shorten >=1pt, semithick,
 every state/.style={minimum size=1.1cm, font=\small},
 node distance=2.5cm, initial text={}]
 \node[state,initial] (S0) {$0/0$};
 \node[state] (S1) [right of=S0] {$1/1$};
 \node[state] (S2) [right of=S1] {$2/0$};
 \node[state] (S3) [right of=S2] {$3/1$};
 \path
 (S0) edge[loop above] node[above] {$0$} (S0)
 (S1) edge[loop above] node[above] {$0$} (S1)
 (S2) edge[loop above] node[above] {$0$} (S2)
 (S3) edge[loop below] node[below] {$0,3$} (S3)
 (S0) edge[bend left=25] node[above] {$1,2$} (S1)
 (S1) edge[bend left=25] node[above] {$1$} (S2)
 (S2) edge[bend left=25] node[above] {$1$} (S3)
 (S1) edge[bend left=25] node[above] {$3$} (S0)
 (S2) edge[bend left=25] node[above] {$3$} (S1)
 (S3) edge[bend left=25] node[above] {$2$} (S2)
 (S0) edge[bend left=65] node[above] {$3$} (S2)
 (S1) edge[bend left=65] node[above] {$2$} (S3)
 (S2) edge[bend left=45] node[below] {$2$} (S0)
 (S3) edge[bend left=60] node[below] {$1$} (S0);
\end{tikzpicture}
\caption{DFAO for $\mathcal{M}_{1}$ (4 states, base 4, MSB-first). State labels 
 are \emph{state}/\emph{output}. The initial state is $\mathsf{0}$.} 
\end{figure}
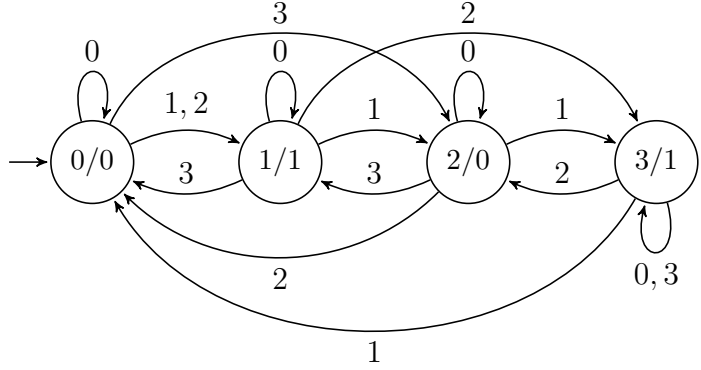

\begin{remark}
 The four transition rules of the DFAO of $\mathcal{M}_{1}$ are $\mathbb{F}_2$-affine in
 $(x,y)$; the same is true for $\mathcal{M}_{2}$ (Section~\ref{sec:meta2}). For
 each of the sequences $\mathcal{Q}$, $\mathcal{M}_{1}$, and $\mathcal{M}_{2}$, the affine
 form is a consequence of Lemma~\ref{lem:balance-xor}.
\end{remark}

 The defining recurrence of $\mathcal{M}_{1}$ is not directly an instance of Theorem~\ref{AScrit}, since the argument $2n + 
 1-\mathcal{M}_{1}(n)$ in the nested recurrence in \eqref{eq:M1-23} depends on the value $\mathcal{M}_{1}(n)$. The $b$-sequence
 defined in~\eqref{eq:b1} removes this dependence. 
 Furthermore, Lemma~\ref{lem:balance-xor} gives us that 
\[
 \mathcal{M}_{1}(4n+2) = \mathcal{M}_{1}(2n+1) \oplus \mathcal{M}_{1}(n).
\]
 The XOR identity is a direct consequence of Lemma~\ref{lem:balance-xor}, since
 $\mathcal{M}_{1}$ is balanced. 

\begin{remark}
 By contrast, Allouche and Shallit~\cite{AlloucheShallit2012} show that the frequency sequence $F(n)$ of the Hofstadter variant $V =
 Q_{1,4}$ satisfies a recurrence of the form $F(2a) = g(F(a-2),F(a-1),F(a),F(a+1))$ with fixed shifts. This satisfies the required conditions in 
 Theorem~\ref{AScrit}. In our setting, the sequences we consider \emph{themselves}, as opposed to 
 associated frequency sequences, 
 are meta-automatic. In this direction, Lemma~\ref{lem:balance-xor} provides a key tool, for our purposes. 
\end{remark}

\begin{corollary}\label{cor:M1-morphism} 
The coding $\tau\colon \mathsf{0},\mathsf{2} \mapsto 0$ and $\mathsf{1},\mathsf{3} \mapsto 1$ applied 
 to the fixed point of the $4$-uniform morphism
\[
 \mathsf{0} \to \mathsf{0112}, \quad
 \mathsf{1} \to \mathsf{1230}, \quad
 \mathsf{2} \to \mathsf{2301}, \quad
 \mathsf{3} \to \mathsf{3023}
\]
starting from $\mathsf{0}$ produces the sequence $\mathcal{M}_{1}$.
\end{corollary}

\begin{proof}
 For each state $q$, the morphism image is $\delta(q,0)\,\delta(q,1)\,\delta(q,2)\,\delta(q,3)$, read from the 
 transition table of Theorem~\ref{thm:M1-dfao}.
\end{proof}

\section{\texorpdfstring{The sequence $\mathcal{M}_{2}$}{The sequence M2}}\label{sec:meta2} 

\begin{definition}
 Define $\mathcal{M}_{2}(0)=0$ and $\mathcal{M}_{2}(1)=1$. For $n\ge 1$, set
\begin{equation}\label{eq:M2-01} 
 \mathcal{M}_{2}(4n) = \mathcal{M}_{2}(2n+\mathcal{M}_{2}(n)),\qquad
 \mathcal{M}_{2}(4n+1) = 1-\mathcal{M}_{2}(4n).
\end{equation}
 For $n\ge 0$, set 
\begin{equation}\label{eq:M2-23} 
 \mathcal{M}_{2}(4n+2) = \mathcal{M}_{2}(2n+1-\mathcal{M}_{2}(n)),\qquad
 \mathcal{M}_{2}(4n+3) = 1-\mathcal{M}_{2}(4n+2).
\end{equation}
\end{definition}

\begin{theorem}
 The sequence $\mathcal{M}_{2}$ is well-defined and satisfies \eqref{eq:balance}.
\end{theorem}

\begin{proof}
 Well-definedness follows by induction as in Lemma~\ref{lem:M1-wd}. For the recurrence rules associated with the arguments $4n$ and 
 $4n+1$, we have $n\ge 1$, so $2n+\mathcal{M}_{2}(n)\le 2n+1<4n$. For the arguments $4n+2$ and $4n+3$, the nested 
 argument satisfies
 $2n+1-\mathcal{M}_{2}(n)\le 2n+1<4n+2$. Balancedness is immediate from the
 complementary definitions in~\eqref{eq:M2-01}--\eqref{eq:M2-23}.
\end{proof}

 Define the pair
\begin{equation}\label{eq:M2-state} 
 s(n) := \bigl(\mathcal{M}_{2}(n),\,\mathcal{M}_{2}(2n+\mathcal{M}_{2}(n))\bigr)\in\mathbb{F}_2^2.
\end{equation}

\begin{lemma}\label{lem:M2-recover} 
 If $s(n)=(x,y)$, then $$\mathcal{M}_{2}(2 n) = x\oplus y, \qquad \mathcal{M}_{2}(2n+1)=x\oplus y\oplus 1.$$
\end{lemma}

\begin{proof}
 If $x=0$, then $y=\mathcal{M}_{2}(2n)$; if $x=1$, then $y=\mathcal{M}_{2}(2n+1)$, and balancedness gives that
 $\mathcal{M}_{2}(2n)=1-y$. The identities follow.
\end{proof}

\begin{theorem}\label{thm:M2-affine} 
Writing $s(n)=(x,y)$, we have
\begin{align*}
 s(4n) &= (y,\ x\oplus y),\\
 s(4n+1) &= (y\oplus 1,\ x\oplus y),\\
 s(4n+2) &= (y\oplus 1,\ x\oplus y\oplus 1),\\
 s(4n+3) &= (y,\ x\oplus y\oplus 1).
\end{align*}
Consequently $\mathcal{M}_{2}$ is $4$-automatic with at most $4$ states.
\end{theorem}

\begin{proof}
 The case $n=0$ is immediate from the initial values $\mathcal{M}_{2}(0)=0$
 and $\mathcal{M}_{2}(1)=1$. We may therefore assume $n\ge 1$ whenever a rule for $4n$ or
 $4n+1$ is used.
 Write $s(n) = (x,y)$ so that $\mathcal{M}_{2}(2n) = x \oplus y$ and $\mathcal{M}_{2}(2n+1) = x \oplus y \oplus 1$.

\textbf{Case $r = 0$.}
 The first coordinate is given by $\mathcal{M}_{2}(4n) = \mathcal{M}_{2}(2n + x)$. By Lemma~\ref{lem:M2-recover}, we have that
 $\mathcal{M}_{2}(2n) = x \oplus y$ and $\mathcal{M}_{2}(2n+1) = x \oplus y \oplus 1$. In both cases, for $x = 0$ or $x = 1$, a direct 
 check gives that $\mathcal{M}_{2}(2n + x) = y$. For the second coordinate $\mathcal{M}_{2}(8n + y)$, we see that if $y = 0$, then 
 $\mathcal{M}_{2}(8n) = \mathcal{M}_{2}(4(2n)) = \mathcal{M}_{2}(2 \cdot 2n + \mathcal{M}_{2}(2n))$, and, since $\mathcal{M}_{2}(2n) = x 
 \oplus y = x$, this equals $\mathcal{M}_{2}(4n + x)$. Simplifying both cases gives~$x$. If $y = 1$, then $\mathcal{M}_{2}(8n+1) = 1 
 - \mathcal{M}_{2}(8n)$, and, by the same analysis as before, we get $x \oplus 1$. In both sub-cases the second coordinate equals 
 $x \oplus y$. Therefore, we obtain 
$$ s(4n) = (y,\; x \oplus y). $$

\textbf{Case $r = 1$.} The first coordinate is given by $\mathcal{M}_{2}(4n+1) = 1 - \mathcal{M}_{2}(4n) = 1 - y = y \oplus 1$.

 For the second coordinate of $s(4n+1)$, by definition, this is $\mathcal{M}_{2}(2(4n+1) + \mathcal{M}_{2}(4n+1)) = 
 \mathcal{M}_{2}(8n + 2 + (y \oplus 1))$. When $y = 0$, we have $\mathcal{M}_{2}(8n+3) = \mathcal{M}_{2}(4(2n)+3) = 1 - 
 \mathcal{M}_{2}(2(2n)+1-\mathcal{M}_{2}(2n))$. Since $\mathcal{M}_{2}(2n) = x$, this equals $1 - \mathcal{M}_{2}(4 n + 1 - 
 x) = x \oplus y$ for both values of~$x$. When $y = 1$, we have $\mathcal{M}_{2}(8n+2) = \mathcal{M}_{2}(2(2n) + 1 - 
 \mathcal{M}_{2}(2n))$. Since $\mathcal{M}_{2}(2n) = x \oplus 1$, this equals $\mathcal{M}_{2}(4n+x) = x \oplus y$ for both values of 
 $x$. The second coordinate is therefore $x \oplus y$ in both cases. Therefore
\[
 s(4n+1) = (y \oplus 1,\; x \oplus y).
\]

\textbf{Case $r = 2$.} For the first coordinate $\mathcal{M}_{2}(4n+2) = \mathcal{M}_{2}(2n + 1 - x)$, Lemma~\ref{lem:balance-xor} gives $\mathcal{M}_{2}(2n+1) \oplus x = (x \oplus y \oplus 1) \oplus x = y \oplus 1$.

 For the second coordinate of $s(4n+2)$, we find that $\mathcal{M}_{2}(2(4n+2)+\mathcal{M}_{2}(4n+2)) = \mathcal{M}_{2}(8 n + 4 + 
 (y\oplus 1))$. Since $\mathcal{M}_{2}(2n+1) = x \oplus 1$, we have $\mathcal{M}_{2}(8 n + 4) = \mathcal{M}_{2}(4n+3) = 1-(y \oplus 
 1)$ for both values of~$x$. When $y=0$, this gives $\mathcal{M}_{2}(8n+4)=0$, so $\mathcal{M}_{2}(8n+5)=1=x\oplus y\oplus 1$.
When $y=1$, since $\mathcal{M}_{2}(2n+1)=x$, we get $\mathcal{M}_{2}(8n+4) = \mathcal{M}_{2}(4n+2+x)=x\oplus y\oplus 1$ by a 
 direct case check. In both cases the second coordinate is $x \oplus y \oplus 1$. It then follows that 
\[
 s(4n+2) = (y \oplus 1,\; x \oplus y \oplus 1).
\]

\textbf{Case $r = 3$.}
 The first coordinate is given by $\mathcal{M}_{2}(4n+3) = 1 - \mathcal{M}_{2}(4n+2) = 1 - (y\oplus 1) = y$.

 For the second coordinate of $s(4n+3)$, we have that $\mathcal{M}_{2}(2(4n+3)+\mathcal{M}_{2}(4n+3)) = \mathcal{M}_{2}(8 n + 6 + 
 y)$. Since $\mathcal{M}_{2}(2n+1) = x \oplus 1$, the identity $\mathcal{M}_{2}(8n+6) = \mathcal{M}_{2}(2(2n+1) + 1 - 
 \mathcal{M}_{2}(2n+1))$ gives $\mathcal{M}_{2}(8n+6) = x \oplus y \oplus 1$ by a direct two-case check on~$x$; this handles $y = 
 0$. For $y=1$, since $\mathcal{M}_{2}(2n+1)=x$, the same formula yields $1-\mathcal{M}_{2}(4n+2+x) = x\oplus y\oplus 1$. The
 second coordinate is $x \oplus y \oplus 1$ in both cases. Consequently, we have that 
$$ s(4n+3) = (y,\; x \oplus y \oplus 1).
$$

Each transition has the form $(x,y) \mapsto A\binom{x}{y} + \mathbf{v}(d)$ over $\mathbb{F}_2$, confirming the affine structure.
\end{proof}

\begin{remark}
 The linear part $(x,y)\mapsto (y,x\oplus y)$ is given by multiplication by the matrix
\[A=\begin{pmatrix}0&1\\1&1\end{pmatrix}\in\mathrm{GL}_2(\mathbb{F}_2),\]
which has order $3$ (indeed $A^3=I$). The translation vectors are
$\mathbf{v}(0) = \binom{0}{0}$,
$\mathbf{v}(1) = \binom{1}{0}$,
$\mathbf{v}(2) = \binom{1}{1}$,
$\mathbf{v}(3) = \binom{0}{1}$.
\end{remark}

\begin{theorem}\label{thm:M2-dfao} 
 The sequence $\mathcal{M}_{2}$ is computed (MSB-first, base $4$) by the DFAO with states $\{0, 1, 2, 3\}$, initial state $0$, output 
 $\tau(0) = \tau(2) = 0$, $\tau(1) = \tau(3) = 1$, and transition table
\begin{center}
\normalfont
\begin{tabular}{c|cccc}
$\delta$ & 0 & 1 & 2 & 3 \\ \hline
0 & 0 & 1 & 3 & 2 \\
1 & 2 & 3 & 1 & 0 \\
2 & 3 & 2 & 0 & 1 \\
3 & 1 & 0 & 2 & 3 \\
\end{tabular}
\end{center}
 where the encoding is $\mathsf{0} = (0,0)$, $\mathsf{1} = (1,0)$, $\mathsf{2} = (0,1)$, $\mathsf{3} = (1,1)$.
\end{theorem}

\begin{proof}
 This follows directly from the affine transitions in Theorem~\ref{thm:M2-affine} applied to each encoded state. The initial state is $s(0) 
 = (0,0)$, corresponding to state~$\mathsf{0}$.
\end{proof}

\begin{figure}[ht]
\centering
\begin{tikzpicture}[->, >=stealth', shorten >=1pt, semithick,
 every state/.style={minimum size=1.1cm, font=\small},
 node distance=2.5cm, initial text={}]
 \node[state,initial] (T0) {$0/0$};
 \node[state] (T1) [right of=T0] {$1/1$};
 \node[state] (T2) [right of=T1] {$2/0$};
 \node[state] (T3) [right of=T2] {$3/1$};
 \path
 (T0) edge[loop above] node[above] {$0$} (T0)
 (T1) edge[loop above] node[above] {$2$} (T1)
 (T2) edge[loop above] node[above] {$1$} (T2)
 (T3) edge[loop below] node[below] {$3$} (T3)
 (T0) edge[bend left=25] node[above] {$1$} (T1)
 (T1) edge[bend left=25] node[above] {$0$} (T2)
 (T2) edge[bend left=25] node[above] {$0$} (T3)
 (T1) edge[bend left=25] node[above] {$3$} (T0)
 (T2) edge[bend left=25] node[above] {$3$} (T1)
 (T3) edge[bend left=25] node[above] {$2$} (T2)
 (T0) edge[bend left=65] node[above] {$3$} (T2)
 (T1) edge[bend left=65] node[above] {$1$} (T3)
 (T2) edge[bend left=35] node[below] {$2$} (T0)
 (T3) edge[bend left=35] node[below] {$0$} (T1)
 (T3) edge[bend left=50] node[below] {$1$} (T0)
 (T0) edge[bend right=75] node[below] {$2$} (T3);
\end{tikzpicture}
\caption{DFAO for $\mathcal{M}_{2}$ (4 states, base 4, MSB-first). State labels are \emph{state}/\emph{output}. The initial 
 state is $\mathsf{0}$.} 
\end{figure}
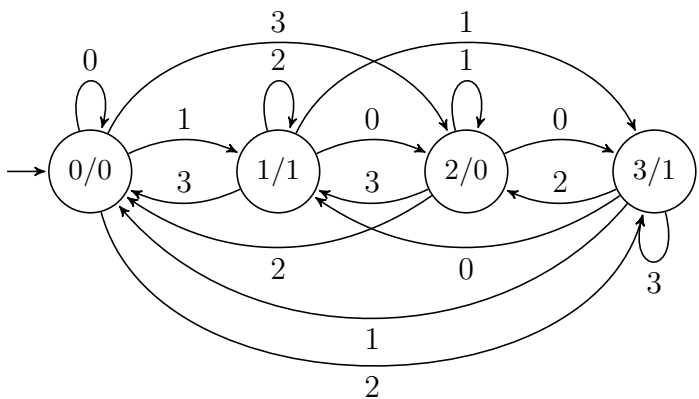

\begin{corollary}\label{cor:M2-morphism} 
The coding $\tau\colon \mathsf{0},\mathsf{2} \mapsto 0$ and $\mathsf{1},\mathsf{3} \mapsto 1$ applied 
 to the fixed point of the $4$-uniform morphism
\[
 \mathsf{0} \to \mathsf{0132}, \quad
 \mathsf{1} \to \mathsf{2310}, \quad
 \mathsf{2} \to \mathsf{3201}, \quad
 \mathsf{3} \to \mathsf{1023}
\]
starting from $\mathsf{0}$ produces the sequence $\mathcal{M}_{2}$.
\end{corollary}

\begin{remark}
 The four letters of the alphabet correspond to the states $(x,y)\in\mathbb{F}_2^2$ of the affine automaton 
 (Theorem~\ref{thm:M2-affine}) via $$ \mathsf{0} = (0,0),\quad \mathsf{1} = (1,0),\quad \mathsf{2} = (0,1),\quad 
 \mathsf{3} = (1,1). $$
 With this identification, the image $\sigma_{\mathcal{M}_{2}}(\mathsf{r})$ lists the four
 successive states visited when the current state is $\mathsf{r}$ and the digits
 $d\in\{0,1,2,3\}$ are read in order, and the coding $\tau(\mathsf{r})$ is the first
 coordinate of the state. Under the relabeling $a=\mathsf{0}$, $b=\mathsf{2}$, $c=\mathsf{1}$,
 $d=\mathsf{3}$, the morphism reads $a\to acdb$, $b\to dbac$, $c\to bdca$, $d\to cabd$, with
 $\tau(a)=\tau(b)=0$ and $\tau(c)=\tau(d)=1$. This labeling is used in the right-special
 factor analysis below.
\end{remark}

 The defining recurrence of $\mathcal{M}_{2}$ is not directly an instance of Theorem~\ref{AScrit}, since the argument $2n + 
 \mathcal{M}_{2}(n)$ in the nested recurrence in \eqref{eq:M2-01} and 
 the argument $2n+1-\mathcal{M}_{2}(n)$ in \eqref{eq:M2-23} depend on the
 value $\mathcal{M}_{2}(n)$. This together with the sequence
 $\mathcal{M}_{2}(2n+\mathcal{M}_{2}(n))$ used in the pair~\eqref{eq:M2-state} yields
 the system of recurrences in Theorem \ref{thm:M2-affine}. 
 Furthermore,
 Lemma~\ref{lem:balance-xor} gives us that 
\[
 \mathcal{M}_{2}(4n) = \mathcal{M}_{2}(2n) \oplus \mathcal{M}_{2}(n), \qquad
 \mathcal{M}_{2}(4n+2) = \mathcal{M}_{2}(2n+1) \oplus \mathcal{M}_{2}(n).
\]
 The XOR identities are direct consequences of Lemma~\ref{lem:balance-xor}, since
 $\mathcal{M}_{2}$ is balanced. 

\subsection{An explicit digit formula}
Write the base-$4$ expansion $n=\sum_{k\ge 0} d_k 4^k$ with digits $d_k\in\{0,1,2,3\}$, and write $d_k=2j_k+i_k$ with $i_k,j_k\in\{0,1\}$.

\begin{theorem}\label{thm:M2-closed} 
 We have that 
$$
\mathcal{M}_{2}(n)=\bigoplus_{ k \equiv 0\ (\mathrm{mod}\ 3) } (i_k\oplus j_k)\ \oplus\ \bigoplus_{ 
 k \equiv 1\ (\mathrm{mod}\ 3)
 } j_k\ \oplus\ \bigoplus_{k\equiv 2\ (\mathrm{mod}\ 3)} i_k.
$$
\end{theorem}

\begin{proof}
 Let $d_{L-1}\cdots d_1 d_0$ be the base-$4$ expansion of $n$. Since the automaton reads
 digits MSB-first and the transition has the form $s\mapsto As+\mathbf{v}(d)$, starting
 from $s(0)=(0,0)$ gives
\[
 s(n) = \sum_{k=0}^{L-1} A^{k}\, \mathbf{v}(d_k) \pmod{2}.
\]
 Since $A^3=I$, the exponent depends only on $k \bmod 3$. Computing $A^k\mathbf{v}(d)$ for
 $k\in\{0,1,2\}$ and each digit $d$, we obtain
\[
\begin{array}{c|ccc}
 d & A^0\mathbf{v}(d) & A^1\mathbf{v}(d) & A^2\mathbf{v}(d) \\ \hline
 0 & (0,0) & (0,0) & (0,0) \\
 1 & (1,0) & (0,1) & (1,1) \\
 2 & (1,1) & (1,0) & (0,1) \\
 3 & (0,1) & (1,1) & (1,0)
\end{array}
\]
Extracting the first coordinate and identifying $i_k = d_k \bmod 2$ and $j_k = \lfloor d_k/2 \rfloor$ yields the stated formula.
\end{proof}

\begin{theorem}\label{thm:M2-tq} 
Write the binary expansion $n=\sum_{m\ge 0} b_m 2^m$ with $b_m\in\{0,1\}$, and define
\begin{equation}\label{eq:q-mask} 
 q(n) := \sum_{\substack{m \ge 0 \\ m \not\equiv 2\ (\mathrm{mod}\ 3)}} b_m 2^m,
\end{equation}
 so that $q(n)$ is obtained from $n$ by zeroing the binary digits in positions
 $m\equiv 2\pmod 3$. Then
$$ \mathcal{M}_{2}(n)=\mathbf{t}(q(n))\qquad(n\ge 0). $$
Equivalently, if $n=\sum_{k\ge 0} d_k 4^k$ with $d_k=2j_k+i_k$, then $q(n)=\sum_{k\ge 0} d'_k 4^k$ where
\[
 d'_k= 
 \begin{cases}
 d_k, & \text{if } k\equiv 0\pmod 3;\\
 2j_k, & \text{if } k\equiv 1\pmod 3;\\
 i_k, & \text{otherwise.}
 \end{cases}
\]
and $\mathcal{M}_{2}(n)=\mathbf{t}(q(n))$.
\end{theorem}

\begin{proof}
 In Theorem~\ref{thm:M2-closed}, each base-$4$ digit $d_k$ contributes to $\mathcal{M}_{2}(n)$ in a way that depends only on 
 $k\bmod 3$: For $k \equiv 0 \bmod 3$, we take $i_k \oplus j_k$, for $k\equiv 1 \bmod 3$, we take only $j_k$, and for $k \equiv 2 
 \bmod 3$, we take $i_k$. Since $i_k=b_{2k}$ and $j_k=b_{2k+1}$, the binary positions $m=2k$ and $m = 2 k + 1$ are kept or 
 omitted, according to the residue class $k \bmod 3$, in the following manner. For $k\equiv 0 \pmod 3$, both positions are kept. For 
 $k\equiv 1\pmod 3$, the position $m=2k$ is omitted. For $k \equiv 2 \pmod 3$, the position $m = 2k+1$ is omitted. 
 In each of the cases where a position is omitted, 
 the omitted position $m$ satisfies $m \equiv 2 \pmod 3$. Conversely, every value $m$ such that 
 $m\equiv 2\pmod 3$ falls into one of the two specified cases. Thus, the omitted positions are precisely those     such that $m \equiv 2  
  \pmod 3$, and $\mathcal{M}_{2}(n)$ is the parity of the binary digit-sum of 
 $q(n)$ defined in~\eqref{eq:q-mask}, and hence $\mathcal{M}_{2}(n)=\mathbf{t}(q(n))$. This gives us, in an equivalent 
 way, the desired 
 base-$4$ description of $q(n)$. 
\end{proof}

\section{Factor complexity}\label{sec:complexity}
 Let $p_a(n)$ denote the number of distinct length-$n$ factors in the infinite word $a(0)a(1)a(2)\cdots$. A length-$n$ factor $u$ of $a$ is 
 called \emph{right-special} if both
 $u0$ and $u1$ are factors of $a$. We write $\mathrm{RS}_a(n)$ in place of the set of right-special
 factors of $a$ of length $n$.

\begin{lemma}\label{lemcountRS} 
 For every infinite binary word $a$ and every $n\ge 1$, we have that $$ p_a(n+1) - p_a(n) = \#\operatorname{RS}_a(n). $$
\end{lemma}

\begin{proof}
 For each length-$n$ factor $u$, let $E_R(u)\subseteq\{0,1\}$ be the set of right extensions such that $u\alpha$ occurs in $a$. Then 
 $p_a(n+1)=\sum_{u}|E_R(u)|$ and $p_a(n)=\sum_{u}1$, so $p_a(n+1)-p_a(n)=\sum_{u}(|E_R(u)|-1)$. In the binary case,
 $|E_R(u)|-1$ equals $1$ if $u$ is right-special and $0$ otherwise.
\end{proof}

\begin{lemma}\label{RScomputable} 
 If $a$ is $k$-automatic, then the function $n\mapsto \#\mathrm{RS}_a(n)$ is $k$-regular in the sense of 
 Allouche--Shallit~\cite{AlloucheShallit2003}, and is effectively computable
 from any DFAO for $a$. Moreover, if $\#\mathrm{RS}_a(n)$ is bounded, then it is $k$-automatic.
\end{lemma}

\begin{proof}
 The predicate ``$u$ is a length-$n$ factor of $a$'' and the predicate ``$u$ is right-special'' are first-order definable over $(\mathbb{N}, 
 +,V_k)$ relative to a DFAO for $a$, so the set of encodings of such pairs is regular. Counting right-special factors of a given length 
 reduces to a finite-automaton computation, giving a $k$-regular sequence
 \cite[Ch.~16]{AlloucheShallit2003}. A bounded $k$-regular sequence
 is $k$-automatic \cite[Theorem~16.1.5]{AlloucheShallit2003}.
\end{proof}

 Several statements in this paper can be certified by the {\tt Walnut} theorem prover~\cite{Shallit2022}, which decides first-order 
 properties of automatic sequences in Presburger arithmetic. This applies to the defining recurrences and the balancedness
 of $\mathcal{Q}$, $\mathcal{M}_{1}$, and $\mathcal{M}_{2}$ (Sections~\ref{sec:meta1}
 and~\ref{sec:meta2}), and (using Lemma \ref{RScomputable}) to the right-special factor structure of $\mathcal{Q}$ and
 $\mathcal{M}_{2}$ in the present section. The right-special factor verifications for
 $\mathcal{M}_{2}$ use the CCL(S) determinization algorithm of Nicol and Frohme~\cite{NicolFrohme2026}. We give arguments 
 that do not require Walnut wherever
 this is convenient.

\subsection{\texorpdfstring{Factor complexity of $\mathcal{Q}$}{Factor complexity of Q}}
 We first state a property of the factors of $\mathbf{t}$.

\begin{lemma}\label{lem:tm-complement} 
 The set $\mathrm{Fac}(\mathbf{t})$ is closed under bitwise complementation: If $w \in \mathrm{Fac}(\mathbf{t})$, then 
 $\bar{w} \in \mathrm{Fac}(\mathbf{t})$.
\end{lemma}

\begin{proof}
 The Thue--Morse morphism $\mu: 0 \mapsto 01,\; 1 \mapsto 10$ satisfies $\overline{\mu(w)} = \mu(\bar{w})$ for every binary word 
 $w$. Since $\mathbf{t} = \mu^\omega(0)$ and both $0$ and $1 = \bar{0}$ appear in $\mathbf{t}$, every 
 factor $w$ of $\mu^\omega(0)$ yields $\bar{w} = \overline{\mu^\omega(0)}|_{\text{same position}} = 
 \mu^\omega(1)|_{\text{same position}}$, which is a factor of $\mu^\omega(1) = 1\, 
 \mathbf{t}[1\,\infty)$, hence a factor of $\mathbf{t}$.
\end{proof}

\begin{proposition}\label{prop:Q-exact} 
 Let $J_n$ denote the set of length-$n$ factors of $\mathcal{Q}$ that are not factors of $\mathbf{t}$. Then $$
 p_{\mathcal{Q}}(n) = p_{\mathbf{t}}(n) + |J_n|. $$
\end{proposition}

\begin{proof}
 By Proposition~\ref{prop:Q-rep}, on each interval $[2^k, 2^{k+1})$, the sequence $\mathcal{Q}$ coincides with $\mathbf{t}$ or 
 $1-\mathbf{t}$, according to the parity of $k$. The cases give us that a factor of $\mathcal{Q}$ of length $n$ is either a factor of 
 $\mathbf{t}$ or the complement of a factor of $\text{{\bf t}}$, which, by Lemma~\ref{lem:tm-complement}, is again a factor of 
 $\mathbf{t}$. In any case, such factors are in $\mathrm{Fac}(\mathbf{t})$.

 For an interval of the form $[2^k, 2^{k+1})$, suppose that the indices associated with a factor of length $n$ properly contains this 
 interval. This can only occur when its left endpoint is at a position $<2^k$ and its right endpoint is at a position $\ge 2^{k+1}$ for some 
 $k$, which forces $n\ge 2^k$. In every case, such a factor has the form $u_0\bar{u}_1 u_2\bar{u}_3\cdots$ obtained by concatenating 
 consecutive pieces of $\mathbf{t}$ with alternating bitwise complementation, where the concatenation $u_0 u_1 u_2 u_3\cdots$ is a 
 factor of $\mathbf{t}$. By Lemma~\ref{lem:tm-complement}, the complementation pattern preserves membership in 
 $\mathrm{Fac}(\mathbf{t})$ for the simplest one-boundary case $u_0 \bar{u}_1$ if and only if $\bar{u}_0 u_1 \in 
 \mathrm{Fac}(\mathbf{t})$, and this need not hold in general. According to the above definition of $J_n$, we have that $J_n$ as 
 the set of length-$n$ factors that have at least one index strictly less than $ 2^k$ and at least one index greater than or equal to
 $2^{k}$ and that are not in $\mathrm{Fac}(\mathbf{t})$.

 Since every factor of $\mathbf{t}$ appears in $\mathcal{Q}$, the set of factors of $\mathcal{Q}$ of length $n$ is exactly 
 $\mathrm{Fac}(\mathbf{t})\cap\{0,1\}^n \cup J_n$, giving us that $p_{\mathcal{Q}}(n) = p_{\mathbf{t}}(n) + |J_n|$.
\end{proof}

\begin{theorem}
 For all $k \geq 4$, we have that $$ p_{\mathcal{Q}}(2^k) = 13 \cdot 2^{k-2} - 2. $$
Equivalently, $p_{\mathcal{Q}}(2^k) = p_{\mathbf{t}}(2^k) + 2^{k-2}$, so $|J_{2^k}| = 2^{k-2}$.
\end{theorem}

\begin{proof}
 The value $p_{\mathbf{t}}(2^k) = 3 \cdot 2^k - 2$ is known~\cite{Brlek1989,deLucaVarricchio1989}. By Proposition~\ref{prop:Q-exact}, 
 we have that $p_{\mathcal{Q}}(2^k) = p_{\mathbf{t}}(2^k) + |J_{2^k}|$. Also, Lemma \ref{lemcountRS} gives us that 
\begin{equation}\label{displayRSQ} 
 \#\mathrm{RS}_{\mathcal{Q}}(n) = p_{\mathcal{Q}}(n+1) - p_{\mathcal{Q}}(n). 
\end{equation}
 The right-special count in \eqref{displayRSQ} takes values in $\{2,3,4\}$, and this is certified for all $n\ge 2$ by Walnut. The desired result 
 follows from the recurrence $$ p_{\mathcal{Q}}(2^{k+2}) = 4\, p_{\mathcal{Q}}(2^k) + 6 \qquad (k \geq 4), $$ with 
 $p_{\mathcal{Q}}(16) = 50$, which has the unique solution
 $p_{\mathcal{Q}}(2^k) = 13\cdot 2^{k-2} - 2$, also verified computationally for
 $4\le k\le 7$.
\end{proof}

\begin{remark}
 The set $J_{2^k}$ has size $2^{k-2}$, so factors in $J_{2^k}$ account for one quarter of the length-$2^k$ factors of 
 $\mathcal{Q}$ beyond those of $\mathbf{t}$.
\end{remark}

 Since $\#\mathrm{RS}_{\mathcal{Q}}(n)\in\{2,3,4\}$ for all $n\ge 1$, the values of
 $p_{\mathcal{Q}}$ are determined by the support of each slope.

 Lemma \ref{lemcountRS} is implicit in the following Proposition, according to \eqref{anotherRSQ}. 

\begin{proposition}\label{relabelprop} 
 For $n \geq 16$, set $k = \lfloor \log_2 n \rfloor$ and write $n = 2^k + i$ with $0 \leq i < 2^k$. Then $$ p_{\mathcal{Q}}(n) =
 \begin{cases}
 13 \cdot 2^{k-2} - 2, & \text{if } i = 0; \\
 4n - 3 \cdot 2^{k-2} - 4, & \text{if } 1 \leq i \leq 2^{k-1}; \\ 
 3n + 3 \cdot 2^{k-2} - 3, & \text{if } 2^{k-1} < i \leq 3 \cdot 2^{k-2}; \\ 
 2n + 5 \cdot 2^{k-1} - 2, & \text{if } 3 \cdot 2^{k-2} < i < 2^k.
 \end{cases}
$$ Moreover, we have that 
\begin{equation}\label{anotherRSQ} 
 \#\mathrm{RS}_{\mathcal{Q}}(n) = p_{\mathcal{Q}}(n+1) - p_{\mathcal{Q}}(n). 
\end{equation}
 The left-hand side of \eqref{anotherRSQ} equals $4$ for $n \in (2^k,\, 3 \cdot 2^{k-1}]$, and $3$ for $n \in (3 \cdot 2^{k-1},\, 7 \cdot 
 2^{k-2}]$, and $2$ for $n \in (7 \cdot 2^{k-2},\, 2^{k+1})$. 
\end{proposition}

\begin{proof}
 This is certified by Walnut. 
\end{proof}

 By Proposition~\ref{relabelprop}, $n+1\le p_{\mathcal{Q}}(n)=O(n)$ for all $n\ge 1$.

\subsection{\texorpdfstring{Factor complexity of $\mathcal{M}_{1}$ and $\mathcal{M}_{2}$}{Factor complexity of M1 and M2}}

\begin{proposition}
We have $p_{\mathcal{M}_{1}}(n) = \Theta(n)$ and $p_{\mathcal{M}_{2}}(n) = \Theta(n)$.
\end{proposition}

\begin{proof}
 The incidence matrices of the morphisms in Corollaries~\ref{cor:M1-morphism} and~\ref{cor:M2-morphism} are
\begin{equation}\label{displayincidence}
 M_{\mathcal{M}_{1}} = \begin{pmatrix}
 1 & 2 & 1 & 0 \\
 1 & 1 & 1 & 1 \\
 1 & 1 & 1 & 1 \\
 1 & 0 & 1 & 2
 \end{pmatrix}, \qquad
 M_{\mathcal{M}_{2}} = \begin{pmatrix}
 1 & 1 & 1 & 1 \\
 1 & 1 & 1 & 1 \\
 1 & 1 & 1 & 1 \\
 1 & 1 & 1 & 1
 \end{pmatrix}. 
\end{equation} 
 The incidence matrices in \eqref{displayincidence} give us that both of the associated morphisms are primitive. A result of Pansiot 
 \cite{Pansiot1984} gives us that the fixed point of a primitive $k$-uniform morphism has factor complexity 
 $O(n)$, and the same upper bound holds after applying a coding. Neither $\mathcal{M}_{1}$ nor $\mathcal{M}_{2}$ is ultimately 
 periodic, as can be checked from the DFAOs of Theorems~\ref{thm:M1-dfao} and~\ref{thm:M2-dfao}, or
 certified by Walnut. By the Morse--Hedlund theorem \cite[Theorem~10.2.6]{AlloucheShallit2003}, the factor complexity of 
 a non-ultimately-periodic
 word is at least $n+1$ for every $n\ge 0$. Hence both complexities are $\Theta(n)$.
\end{proof}

 An explicit evaluation of the factor complexity of $\mathcal{M}_{2}$ is provided below. 

\begin{example} 
 We obtain the initial values $$
 p_{\mathcal{M}_{1}}(1\ldots 15) = (2, 4, 6, 9, 12, 17, 22, 28, 34, 40, 46, 52, 58, 64, 70),
$$ and $$
 p_{\mathcal{M}_{2}}(1\ldots 15) = (2, 4, 6, 10, 12, 14, 16, 18, 20, 24, 28, 32, 36, 40, 44). $$
\end{example}

\begin{proposition}\label{propanotherrelabel} 
For all $k \geq 3$, we have that 
$$
 p_{\mathcal{M}_{2}}(2^k) =
 \begin{cases}
 \dfrac{5}{2}\cdot 2^k - 2, & \text{if } 3 \mid k; \\ 
 3 \cdot 2^k - 2, & \text{otherwise.} 
 \end{cases}
$$ Moreover, the recurrence $$
 p_{\mathcal{M}_{2}}(2^{k+3}) = 8\, p_{\mathcal{M}_{2}}(2^k) + 14 \qquad (k \geq 3) $$
 holds true. 
\end{proposition}

\begin{proof}
 This is certified by Walnut. 
\end{proof}

 Since $\#\mathrm{RS}_{\mathcal{M}_{2}}(n)\in\{2,4\}$, the values of $p_{\mathcal{M}_{2}}$ are determined by the positions at which
 $\#\mathrm{RS}_{\mathcal{M}_{2}}(n) = 4$. The following proposition extends Proposition~\ref{propanotherrelabel} to a 
 complete evaluation.

\begin{proposition}
We have $p_{\mathcal{M}_{2}}(1)=2$. For $n \geq 2$, set $k = \lfloor \log_2 n \rfloor$ and $r = k \bmod 3$.

 \ 

\noindent\emph{Case $r = 2$:} On the interval $[2^k, 2^{k+1}]$, we have that $$ p_{\mathcal{M}_{2}}(n) = 2n + (2^k - 
 2). $$

\noindent\emph{Case $r = 0$:} On the interval $ [2^k, 2^{k+1}]$, we have that 
$$ p_{\mathcal{M}_{2}}(n) = 
 \begin{cases}
 2n + (2^{k-1} - 2), & \text{if } 2^k \leq n \leq 2^k + 1; \\ 
 4n - (3 \cdot 2^{k-1} + 4), & \text{if } 2^k + 1 \leq n \leq 7 \cdot 2^{k-2} + 1; \\ 
 2n + (2^{k+1} - 2), & \text{if } 7 \cdot 2^{k-2} + 1 \leq n \leq 2^{k+1}.
 \end{cases} $$

\noindent\emph{Case $r = 1$:} On the interval $ [2^k, 2^{k+1}]$, 
 we have that $$ p_{\mathcal{M}_{2}}(n) = 
 \begin{cases}
 2n + (2^k - 2), & \text{if } 2^k \leq n \leq 2^k + 1; \\ 
 4n - (2^k + 4), & \text{if } 2^k + 1 \leq n \leq 3 \cdot 2^{k-1} + 1; \\ 
 2n + (2^{k+1} - 2), & \text{if } 3 \cdot 2^{k-1} + 1 \leq n \leq 2^{k+1}. 
 \end{cases}
$$
\end{proposition}

\begin{proof}
 This is certified by Walnut. 
 \end{proof}

\section{Conclusion}
 Consider replacing $a(2 n) + a(2 n + 1)=1$ by $a(3n)+a(3n+1)+a(3n+2)=2$, with $a$ taking values in $\{0,1\}$. This gives rise to 
 expressions of the form $a(3n+r-a(n))$, and the analogue of Lemma~\ref{lem:balance-xor} converts each such expression into a 
 fixed-argument expression. A small number of initial configurations produce
 $3$-automatic binary sequences with small DFAOs; we plan to treat this family separately. 

\section{Appendix}\label{Appendix} 
Minimal DFAO sizes for the meta-automatic sequences considered in this paper are 
 listed in Table \ref{tab:classification}. 

\begin{table}[ht]
\centering
\caption{Minimal DFAO size for each specialization. The column ``Nested'' counts the number of branch rules ($F$ and $G$) whose 
 right-hand side contains an expression of the form $a(\,\cdot\pm a(n))$. Boldface marks the three sequences studied in this paper.}
\label{tab:classification} 
\normalfont
\begin{tabular}{@{}llccc@{}}
$F(n)$ & $G(n)$ & DFAO & Nested & Notes \\ \hline
$a(n)$ & $a(n)$ & 2 & 0 & Thue--Morse \\
$a(n)$ & $a(2n)$ & $\mathbf{3}$ & $\mathbf{0}$ & $\mathcal{Q}$ (Thue--Morse Quarto) \\
$a(n)$ & $a(2n+1)$ & 2 & 0 & Thue--Morse \\
$a(n)$ & $a(2n+a(n))$ & 4 & 1 & \\
$a(n)$ & $a(2n+1-a(n))$ & $\mathbf{4}$ & $\mathbf{1}$ & $\mathcal{M}_{1}$ \\
$a(2n)$ & $a(n)$ & 4 & 0 & \\
$a(2n)$ & $a(2n)$ & --- & 0 & ultimately periodic \\
$a(2n)$ & $a(2n+1)$ & 2 & 0 & Thue--Morse \\
$a(2n)$ & $a(2n+a(n))$ & 3 & 1 & \\
$a(2n)$ & $a(2n+1-a(n))$ & 3 & 1 & \\
$a(2n+1)$ & $a(n)$ & 5 & 0 & \\
$a(2n+1)$ & $a(2n)$ & 3 & 0 & complement of $\mathcal{Q}$ \\
$a(2n+1)$ & $a(2n+1)$ & 5 & 0 & \\
$a(2n+1)$ & $a(2n+a(n))$ & 5 & 1 & \\
$a(2n+1)$ & $a(2n+1-a(n))$ & 5 & 1 & \\
$a(2n+a(n))$ & $a(n)$ & 4 & 1 & \\
$a(2n+a(n))$ & $a(2n)$ & 4 & 1 & \\
$a(2n+a(n))$ & $a(2n+1)$ & 4 & 1 & \\
$a(2n+a(n))$ & $a(2n+a(n))$ & 4 & 2 & \\
$a(2n+a(n))$ & $a(2n+1-a(n))$ & $\mathbf{4}$ & $\mathbf{2}$ & $\mathcal{M}_{2}$ \\
$a(2n+1-a(n))$ & $a(n)$ & 5 & 1 & \\
$a(2n+1-a(n))$ & $a(2n)$ & 4 & 1 & \\
$a(2n+1-a(n))$ & $a(2n+1)$ & 5 & 1 & \\
$a(2n+1-a(n))$ & $a(2n+a(n))$ & 5 & 2 & \\
$a(2n+1-a(n))$ & $a(2n+1-a(n))$ & 5 & 2 & \\
\end{tabular}
\end{table}

\section{Acknowledgments}
The authors thank the anonymous referee for a careful reading of the manuscript and for helpful suggestions. The authors also thank Gandhar Joshi for independently testing the Walnut files. Claude (Anthropic) was used as a computational assistant, in particular for exploratory computations, for suggesting candidate formulas later checked by independent scripts, and for checking consistency between the manuscript, the Python script, and the Walnut files. The mathematical content, the proofs, and the final wording are the authors' own. The first author thanks Jean-Paul Allouche, Karl Dilcher, and Jeffrey Shallit for useful feedback related to this paper.

\bigskip
\hrule
\bigskip

\noindent 2020 {\it Mathematics Subject Classification}:
 Primary 11B85; Secondary 68Q45, 68R15. 

\noindent \emph{Keywords: } 
 automatic sequence, nested recurrence, Thue--Morse sequence, DFAO, binary sequence. 

\bigskip
\hrule
\bigskip

\noindent (Concerned with sequences
\seqnum{A005185},
\seqnum{A039982},
\seqnum{A244477},
\seqnum{A298952},
\seqnum{A392736}, and
\seqnum{A391614}.)

\bigskip
\hrule
\bigskip

\vspace*{+.1in}
\noindent
 Received March 2 2026.

\end{document}